\documentclass[12pt,reqno]{amsart}
\usepackage[utf8]{inputenc}
\usepackage[T1]{fontenc}
\usepackage[usenames, dvipsnames]{color}
\usepackage{ulem}

\usepackage{dsfont, amsfonts, amsmath, amssymb,amscd, stmaryrd, latexsym, amsthm, dsfont}
\usepackage[frenchb,english]{babel}
\usepackage{enumerate}
\usepackage{longtable}
\usepackage{geometry}
\usepackage{tikz}
\usetikzlibrary{shapes,arrows}
\usepackage{float}
\newtheorem{prop}{Proposition}
\newtheorem{theorem}{Theorem}
\newtheorem{corollary}{Corollary}
\newtheorem{lm}{Lemma}

\theoremstyle{remark}
\newtheorem{example}{\bf Example}

\newtheorem{rmk}{\bf Remark}
\def\NN{\mathds{N}}

\def\kk{\mathds{k}}
\def\QQ{\mathbb{Q}}

 \usepackage{amsmath}
 \newcommand{\genlegendre}[4]{%
 	\genfrac{(}{)}{}{#1}{#3}{#4}%
 	\if\relax\detokenize{#2}\relax\else_{\!#2}\fi
 }

\usepackage{hyperref}
\hypersetup{
	colorlinks=true,
	urlcolor=blue,
	citecolor=blue}

\begin{document}

\selectlanguage{english}


 \selectlanguage{english}
 \title[ On some imaginary triquadratic number...]{On some imaginary triquadratic number fields $k$ with
 $\mathrm{Cl}_2(k)\simeq(2, 4)$ or $(2, 2, 2)$}
 \author[A. Azizi]{Abdelmalek Azizi}
 \address{Abdelmalek Azizi: Mohammed First University, Mathematics Department, Sciences Faculty, Oujda, Morocco }
 \email{abdelmalekazizi@yahoo.fr}
 \author[M. M. Chems-Eddin]{Mohamed Mahmoud CHEMS-EDDIN}
 \address{Mohamed Mahmoud CHEMS-EDDIN: Mohammed First University, Mathematics Department, Sciences Faculty, Oujda, Morocco }
 \email{2m.chemseddin@gmail.com}

 \author[A. Zekhnini]{Abdelkader Zekhnini}
 \address{Abdelkader Zekhnini: Mohammed First University, Mathematics Department, Pluridisciplinary faculty, Nador, Morocco}
 \email{zekha1@yahoo.fr}



\maketitle

 \begin{abstract}
 	Let $d$ be a square free integer and $L_d:=\mathbb{Q}(\zeta_{8},\sqrt{d})$. In the present work we determine all the fields $L_d$ such that  the $2$-class group, $\mathrm{Cl}_2(L_d)$, of $L_d$ is of type $(2,4)$ or $(2,2,2)$.
 \end{abstract}

 {\section{Introduction}}
 Let $k$ be a number field and $\mathrm{Cl}_2(k)$ its $2$-class group, that is the $2$-Sylow subgroup  of its ideal class group $\mathrm{Cl}(k)$.
 The problem of  determining the structure of $\mathrm{Cl}_2(k)$  is one of the most interesting problems of algebraic number theory,  accordingly many mathematicians  treated  this problem for some number fields of degree $2$, $4$. For example, in \cite{kaplan76} using  binary quadratic forms theory, Kaplan, determined the $2$-class group of some quadratic number fields.
 The authors of \cite{BFL1,BFL2} used genus theory and class field theory to characterize those imaginary quadratic number fields, $k$, with $2$-class group of type $(2,2^m)$ or $(2,2,2)$ and the $2$-rank of the class group of its Hilbert $2$-class field equal to $2$.
  In  \cite{azmohib}, using units and the $2$-part of the class number of  subextensions of $k$, the authors determined the $2$-class group of some real biquadratic number fields $k=\mathbb{Q}(\sqrt{m}, \sqrt{d})$ with $d$ is an odd square free integer.  Using similar techniques, the paper  \cite{azizitaous(24)(222)} characterizes all the fields  $k=\mathbb{Q}(i, \sqrt{d})$ such that $\mathrm{Cl}_2(k)$ is of type $(2,4)$  or $(2,2,2)$
 (here $(a_1,...,a_r)$ denotes the direct sum of cyclic groups of order $a_i$, for $i=1,...,r$). Whenever $k$ is an imaginary multiquadratic number field, this problem is  strongly related to the units of $k$ and the class number of the $2$-part of the class numbers of   its sub-extensions  as we will see later.
 This paper is, actually,  a continuation and extension of our earlier work \cite{ACZ}, in which we determined the rank of the $2$-class group of  all fields of the form
 $L_d:=\mathbb{Q}(\zeta_{8},\sqrt{d})$ with $d$ a positive square free integer and moreover we determined all  fields $L_d$ for which the $2$-class group, $\mathrm{Cl}_2(L_d)$,  is  of type $(2,2)$. In this work, we are interested in determining all  positive square free integers $d$ satisfying  $\mathrm{Cl}_2(L_d)$ is of type $(2,4)$ or  $(2,2,2)$.

 \section*{Notations}
  The next notations will be used for the rest of this article:
 { \footnotesize \begin{enumerate}[\rm$\bullet$]
 		\item $d$:  A positive  odd square free integer,
 		\item$L_{d}$:  $\mathbb{Q}(\zeta_{8},\sqrt{d})$,
 		\item $\mathrm{Cl}_2(L)$:   The $2$-class group of some number field  $L$,
 		\item $h(L)$:   The class number of the number field $L$,
 		\item $h_2(L)$:   The $2$-class number of the number field $L$,
 		\item $ h_2(m)$:   The $2$-class number of the quadratic  field $\mathbb{Q}(\sqrt m)$,
 		\item $E_{L}$:  The unit group of any number field  $L$,
 		\item $\varepsilon_m$:   The fundamental unit of $\mathbb{Q}(\sqrt{m})$,
 		\item $W_{L}$:   The set of roots of unity contained in a field $L$,
 		\item $\omega_{L}$:   The cardinality of $W_{L}$,
 		\item $L^{+}$:   The maximal real subfield of an imaginary number field $L$,
 		\item $Q_{L}$:   The  Hasse's index, that is  $[E_{L}:W_{L}E_{L^{+}}]$, if $L/L^+$ is CM,
 		\item $Q(L/k)$:   The unit index of a $V_4$-extension  $L/k$,
 		\item $q(L):=[E_{L}: \prod_{i}E_{k_i}]$, with  $k_i$ are  the  quadratic subfields	of $L$.	
 \end{enumerate}}

 \section{Preliminaries}
 Let us first collect some results   and definitions that we will    need in the sequel. Recall that a field $K$ is said to be a CM-field if it is  a totally complex quadratic extension of a totally real number field.
Note also that a $V_4$-extension $K/k$ (i.e., a normal extension of number fields with $\mathrm{Gal}(K/k)=V_4$, where $V_4$ is  the Klein's four-group) is called $V_4$-extension of CM-fields if exactly two of its three quadratic sub-extensions   are CM-fields.
 Let us  next recall the   class number formula for a $V_4$-extension of CM-fields:
 \begin{prop}[\cite{lemmermeyer1995ideal}]\label{ class number f.}
 	Let $L/K$ be a $V_4$-extension of $\mathrm{CM}$-fields, then
 	$$h(L)=\frac{Q_{L}}{Q_{K_1}Q_{K_2}}\cdot\frac{\omega_{L}}{\omega_{K_1}\omega_{K_2} }\cdot\frac{h(K_1)h(K_2)h(L^+)}{h(K)^2}\cdot$$
 	Where $K_1, K_2, L^+$ are the three sub-extensions of $L/K$, with $K_1$ and $K_2$ are  $CM$-fields.
 \end{prop}

The following class number formula for a multiquadratic number field is usually attributed to Kuroda \cite{Ku-50}, but it goes back to Herglotz \cite{He-22}  (cf. \cite[p. 27]{BFL}).
  \begin{prop}\label{wada's f.}
 	Let $K$ be a multiquadratic number field of degree $2^n$, $n\in\mathds{N}$,  and $k_i$ the $s=2^n-1$ quadratic subfields of $K$. Then
 $$h(K)=\frac{1}{2^v}[E_K: \prod_{i=1}^{s}E_{k_i}]\prod_{i=1}^{s}h(k_i),$$
 with $$v=\left\{ \begin{array}{cl}
 	n(2^{n-1}-1); &\text{ if } K \text{ is real, }\\
 	(n-1)(2^{n-2}-1)+2^{n-1}-1 & \text{ if } K \text{ is imaginary.}
 	\end{array}\right.$$
 \end{prop}


Continue with  the next formula called Kuroda's class number formula for a $V_4$-extension $K/k$.
 \begin{prop}[\cite{lemmermeyer1994kuroda},  p. 247]\label{Kurod's f.}
 	Let $K/k$ be a  $V_4$-extension. Then we have :
 	$$h(K)=\left\{ \begin{array}{ll}
 	 \vspace{0.2cm}\frac{1}{4} \cdot Q(K/k)\cdot\prod_{i=1}^{3}h(k_i) &\text{ if } k=\mathbb Q \text{ and } K \text{ is real, }\\
	 \vspace{0.2cm}\frac{1}{2} \cdot Q(K/k)\cdot \prod_{i=1}^{3}h(k_i) & \text{ if } k=\mathbb Q \text{ and } K \text{ is imaginary, }\\
 	\frac{1}{4} \cdot Q(K/k)\cdot \prod_{i=1}^{3}h(k_i)/h(k)^2 & \text{ if } k \text{ is an imaginary  quadratic extension of  } \mathbb {Q}.
 	
 	\end{array}\right.$$
 Where $k_i$ are the $3$ subextensions of $K/k$.	
 \end{prop}


 \section{Fields $L_d$ for which $\mathrm{Cl}_2(L_d)$ is of type $(2,4)$}

 Our goal in this section is to determine all  fields $L_d$ for which $\mathrm{Cl}_2(L_d)\simeq(2,4)$. Recall first the definition of the rational biquadratic residue symbol:

  For a prime $p\equiv 1\pmod 4$ and a quadratic residue $a\pmod p$, $\genfrac(){}{0}{a}{p}_4$     will denote the rational biquadratic residue symbol defined by   $$\genfrac(){}{0}{a}{p}_4=\pm1\equiv a^{\frac{p-1}{4}}\pmod p.$$
  Moreover,  for an integer $a\equiv 1\pmod 8$, the symbol $\genfrac(){}{0}{a}{2}_4$ is defined by
  $$\genfrac(){}{0}{a}{2}_4=1 \text{ if }  a\equiv 1\pmod{16},\;\;\;\;\;\;\;\;\; \genfrac(){}{0}{a}{2}_4=-1 \text{ if }  a\equiv 9\pmod{16}.$$

  It turns out that   $\genfrac(){}{0}{a}{2}_4=(-1)^{\frac{a-1}{8}}$.

 \vspace{0.2cm}In all this section, let $p$, $p_i$, $q$ and $q_i$ be prime   integers such that $p\equiv p_i\equiv1\pmod4$ and $q\equiv q_i\equiv3\pmod4$ with $i\in\mathds{N}^*$. The  following lemma is proved in our earlier paper \cite[Theorem 5.6]{ACZ}.

 \begin{lm}\label{lm 2-rank = 2}
 	The rank of the $2$-class group   $\mathrm{Cl}_2(L_d)$ of $L_{d}$ equals $2$  if and only if $d$  takes one of the following forms:
 	\begin{enumerate}[\rm1.]
 		\item $d=q_1q_2$ with   $q_1\equiv q_2 \equiv3\pmod8$.
 		\item $d=p_1p_2$ with   $ p_1\equiv p_2\equiv5\pmod8$.
 		\item $d=q_1q_2$ with   $q_1\equiv3\pmod8$ and $q_2\equiv7\pmod8$.
 		\item $d=pq$ with   $p\equiv5\pmod8$ and $q\equiv7\pmod8$.
 		\item $d=p$ with   $ p\equiv1\pmod8$ and $\left[p\equiv 9 \pmod{16} \text{ or } \genfrac(){}{0}{2}{p}_4\not=\genfrac(){}{0}{p}{2}_4\right] $.
 	\end{enumerate}
 \end{lm}


We need also the following result.

 \begin{lm}\label{lm h2 is divisible by 16  in 5mod8}
 	Let $d=p_1p_2$, where $p_1, p_2$ are two rational primes such that $p_i\equiv 5\pmod 8$. Then $h_2(L_{d})$ is divisible by $16$.
 \end{lm}
 \begin{proof}
 	We have $L_d=\mathbb Q(\sqrt{p_1p_2}, \sqrt2, i)$ is  an imaginary  multiquadratic number field of degree $2^3$. 	So by Proposition \ref{wada's f.} we have
 	\begin{eqnarray*}
 		h(L_{d})=\frac{1}{2^5}q(L_{d})h( p_1p_2) h(- p_1p_2)h(2p_1p_2)h(-2p_1p_2)h(2)h(-2)h(-1).		
 	\end{eqnarray*}
 It is known that $h_2(2)$, $h_2(-2)$ and  $h_2(-1)$  are equal to $1$, and by \cite[p. 350]{kaplan76} $h_2(-2p_1p_2)=4$. On the other hand,    \cite[Theorem 2.2]{azt2014structure} implies that $q(L_{d})$ is a power of $2$. So by
 passing to the $2$-part in the above equation, we get
 \begin{eqnarray}\label{eq in proof lm 3}
 	h_2(L_{d})	&=&\frac{1}{2^3}q(L_{d})h_2(p_1p_2) h_2(-p_1p_2) h_2(2p_1p_2).
 \end{eqnarray}
Note that:
\begin{enumerate}[\rm $\bullet$]
	\item By \cite[ Corollary 18.4]{connor88}, $h_2(p_1p_2)$ is divisible by $2$.
	\item By    \cite[pp. 348-349]{kaplan76} (Propositions $B_1'$ and  $B_4'$), $h_2(-p_1p_2)$ is divisible by $8$.
		\item By \cite[Corollaries 18.4, 19.7 and 19.8 ]{connor88}, $h_2(2p_1p_2)$ is divisible by $4$.
\end{enumerate}
On one hand, $\zeta_{8}=\frac{1+i}{\sqrt{2}}\in E_{L_d}$. On the other hand, letting  $k_i$ be  the  quadratic subfields	of $L_d$, one gets easily
$$\prod_{i}E_{k_i}=\langle i, \varepsilon_2,\varepsilon_{-2}, \varepsilon_{p_1p_2}, \varepsilon_{2p_1p_2},\varepsilon_{-p_1p_2} ,\varepsilon_{-2p_1p_2}  \rangle.$$
So the $8$-th root of unity $\zeta_{8}\not\in \prod_{i}E_{k_i}$. Thus  $\overline{1}$ and  $\overline{\zeta_{8}}$ are two distinct cosets in the quotient $E_{L_d}/ \prod_{i}E_{k_i}$.
   Thus,  $q(L_{d})$ is divisible by $2$.  It follows  by  the equality \eqref{eq in proof lm 3} above that   $h_2(L_{d})$ is divisible by $\frac{1}{2^3}\cdot 2 \cdot  2\cdot  8\cdot  4=16$ as we wished to prove.
 \end{proof}

\begin{example}
Let $d=p_1p_2$ be as in the above Lemma.	Using  PARI/GP calculator version 2.9.4 (64bit), Dec 20, 2017, for $5\leq p_i \leq 200$, $i=1,2$,  we could not find a field $L_d$ such that $h_2(L_{d})=16$.
		We have the following examples.
	\begin{enumerate}[\rm 1.]
			\item For $p_1=13$ and $p_2=5$, we have $h_2(L_{13.5})=32$.
		\item For $p_1=37$ and $p_2=53$, we have $h_2(L_{37.53})=64$.
	\end{enumerate}
\end{example}

\begin{rmk} With   hypothesis and notations of Lemma \ref{lm h2 is divisible by 16  in 5mod8}, we find in  \cite{azt2014structure}  a unit group of $L_d$.
\end{rmk}

 \begin{lm}[\cite{ZAT-15}, Lemma 2]\label{3:105}
	Let $d\equiv1\pmod4$ be a positive square free integer and   $\varepsilon_d=x+y\sqrt d$ be the fundamental unit of  $\QQ(\sqrt d)$. Assume   $N(\varepsilon_d)=1$, then
	\begin{enumerate}[\rm\indent1.]
		\item $x+1$ and $x-1$ are not squares in  $\NN$, i.e., $2\varepsilon_{d}$ is not a square in  $\QQ(\sqrt{d})$.
		\item For all prime  $p$ dividing   $d$, $p(x+1)$ and $p(x-1)$ are not squares in  $\NN$.
	\end{enumerate}
\end{lm}

 \begin{lm}\label{lm on unit undex}
 	Let $d=q_1q_2$, with   $q_1\equiv q_2 \equiv3\pmod8$ are two primes such that $\left(\dfrac{	q_1}{	q_2}\right)=1$. Then we have $E_{L_d}=\langle \zeta_{8},\varepsilon_{  2},  \varepsilon_{   q_1q_2} , \sqrt{\varepsilon_{   2q_1q_2}}\rangle$, and thus $q(L_d)=4$.
\end{lm}
\begin{proof}
As $q_1q_2\equiv 1\pmod 8$, we claim that the unit $\varepsilon_{q_1q_2}$ can be written as $\varepsilon_{q_1q_2}=a+b\sqrt{q_1q_2}$, where  $a$ and $b$ are two integers. Indeed, suppose that $\varepsilon_{q_1q_2}=({\alpha+\beta \sqrt d})/{2}$ where $\alpha$, $\beta$
are  two integers. 	Since  $N(\varepsilon_{q_1q_2})=1$, one deduces that
$\alpha^2-4=\beta^2d$, hence  $\alpha^2-4\equiv \beta^2\pmod8$. On the other hand, if we suppose that $\alpha$ and $\beta$ are odd, then $\alpha^2\equiv \beta^2\equiv1\pmod8$, but this implies the contradiction $-3\equiv1\pmod8$. Thus $\alpha$ and $\beta$ are even and our claim is established.

It is known that  $N(\varepsilon_{q_1q_2})=1$, so by the unique factorization in $\mathbb{Z}$ and Lemma \ref{3:105}  one gets
$$(1):\ \left\{ \begin{array}{ll}
a+1=2q_1b_1^2\\
a-1=2q_2b_2^2,
\end{array}\right.\quad
\text{ or }\quad
(2):\ \left\{ \begin{array}{ll}
a-1=2q_1b_1^2\\
a+1=2q_2b_2^2,
\end{array}\right.
$$	

for some integers $b_1$ and $b_2$   such that   $b=2b_1b_2$.

If the system $(2)$ holds, then $$-1=\left(\frac{2q_1}{q_2}\right)=\left(\frac{a-1}{q_2}\right)=\left(\frac{a+1- 2}{q_2}\right)= \left(\frac{ -2}{q_2}\right) =1.$$ Which is absurd.
Therefore
$ \left\{ \begin{array}{ll}
a+1=2q_1b_1^2\\
a-1=2q_2b_2^2.
\end{array}\right. $ Thus,  $\sqrt{ \varepsilon_{ q_1q_2}}=b_1\sqrt{q_1}+b_2\sqrt{q_2}$.  So $\varepsilon_{ d}$ is not a square in $L_{d}^+$.

We have $\varepsilon_{2d}$ 	has a positive norm. Put $\varepsilon_{2d}=x+y\sqrt{2q_1q_2}$. We similarly show that
 $\sqrt{2\varepsilon_{2d}}=y_1 +y_2\sqrt{2q_1q_2}\in L_{d}^+$, for some integers $y_1$ and $y_2$.  Therefore, $\varepsilon_{2d}$ is  a square in $L_{d}^+$ since $\sqrt2\in L_{d}^+$.
   As $\varepsilon_{ 2}$ has a negative norm, then using
 the algorithm described in \cite{wada} (or in \cite[p. 113]{azizi99unitedeg8}) we have $\{\varepsilon_{  2},  \varepsilon_{   q_1q_2} , \sqrt{\varepsilon_{   2q_1q_2}}\}$
 	is a fundamental system of unit of
 	$L^+_d$. Hence the result follows  easily  by \cite[Proposition 2]{azizi99unite}.
\end{proof}

 To continue, consider the following parameters:

 \begin{enumerate}[\rm $\bullet$]
 	\item For  a rational prime $p$ such that $p\equiv 1\pmod 8$, set $p=u^2-2v^2$ where $u$ and $v$ are two positive integers  such that $u\equiv 1\pmod8$ (for the existence of $u$ and $v$ see \cite{leonard1982divisibilityby16}).
 	\item For two primes $q_1$ and $q_2$ such that  $q_1\equiv q_2\equiv 3\pmod 8$,   $\genfrac(){}{0}{q_1}{q_2}=1$, there exist five  integers    $X,Y,k,l$ and $m$ such that $2q_2=k^2X^2+2lXY+2mY^2$ and $q_1=l^2-2k^2m$ (cf. \cite[p. 356]{kaplan76}).
 \end{enumerate}

 \begin{theorem}\label{thm (2,4)}
 	Let $d$ be an odd positive square free integer. Then 	$\mathrm{Cl}_2(L_d)\simeq(2, 4)$ if and only if $d$ takes one of the two following forms:
 	\begin{enumerate}[\rm 1.]
 		\item $d=p\equiv 9\pmod{16}$ is a  prime  such that $\genfrac(){}{0}{2}{p}_4\not=\genfrac(){}{0}{p}{2}_4$and  $\genfrac(){}{0}{u}{p}_4 =-1$.
 		\item $d=q_1q_2$, with  $q_1\equiv q_2\equiv 3\pmod 8$ are primes such that  $\genfrac(){}{0}{q_1}{q_2}=1$ and       $\genfrac(){}{0}{-2}{\left|k^2X+lY\right|}=-1$.
 	\end{enumerate}	
 \end{theorem}
 \begin{proof}
 	It suffices to determine for which forms of $d$ appearing in Lemma \ref{lm 2-rank = 2}, we have $h_2(L_d)=8$. Let us firstly eliminate some cases.	
 	\begin{enumerate}[\rm $\bullet$]
 		\item 	By \cite[Propositions 5.13 and 5.14]{ACZ}, $h_2(L_d)\not=8$ whenever $d$ takes  forms in  the third and the fourth item of   Lemma \ref{lm 2-rank = 2}.
 		\item The form of $d$ in the second item of   Lemma \ref{lm 2-rank = 2} is also eliminated 	by Lemma \ref{lm h2 is divisible by 16  in 5mod8}.
 		\item 	Note that if $p\equiv 1\pmod{16} \text{ and  }  \genfrac(){}{0}{2}{p}_4\not= \genfrac(){}{0}{p}{2}_4$, then by  \cite[Theorem 5.7]{ACZ}, $h_2(L_p)\not=8$.
 	\end{enumerate}

 	Note that by  results in  the beginning of
 	Section $3$, $p\equiv 9\pmod{16}$ implies that $\genfrac(){}{0}{p}{2}_4=-1$. Hence, by  Lemma \ref{lm 2-rank = 2},  we have  to check the following cases.
 	\begin{enumerate}[\rm(I)]
 		\item $d=p$, is a prime such that $\genfrac(){}{0}{2}{p}_4= \genfrac(){}{0}{p}{2}_4=-1$.\label{ frm i}
 		\item $d=p$, is a prime such that $p\equiv 9\pmod{16} \text{ and  }  \genfrac(){}{0}{2}{p}_4\not= \genfrac(){}{0}{p}{2}_4$.\label{ frm ii}
 		\item $d=q_1q_2$, for two primes $q_1$ and $q_2$ such that  $q_1\equiv q_2\equiv 3\pmod 8$.\label{ frm iii}
 	\end{enumerate}

Let $p\equiv 1 \pmod 8$ be a prime.
Set  $L_p^+=\mathbb{Q}(\sqrt{2}, \sqrt{p})$, $K=\mathbb{Q}(\sqrt{2}, i)$  and $K'=\mathbb{Q}(\sqrt{2}, \sqrt{-p})$. By applying Proposition \ref{ class number f.} to the extension $L_p/\mathbb{Q}(\sqrt{2})$, we have
\begin{eqnarray*}
 h(L_p)=\frac{Q_{L_p}}{Q_{K}Q_{K'}}\frac{\omega_{L_p}}{\omega_{K} \omega_{K'}}
 \frac{h( L_{p}^+)h(K)h(K')}{h(\mathbb{Q}(\sqrt{2}))^2}\cdot
\end{eqnarray*}

We have $ h(\mathbb{Q}(\sqrt{2}))=h(K)=1$. By \cite[Théorème 3]{taous2008}, $Q_{L_p}=1$ and by \cite[Lemma 2.5]{ACZ} $Q_{K}=1$. Since $ \omega_{L_p}= \omega_{K}=8$ and $ \omega_{K'}=2$, then by passing to the $2$-part in the above equality we get
\begin{eqnarray}\label{eq 2}
h_2(L_p)=\frac{1}{2Q_{K'}}h_2( L_{p}^+)h_2(K')\cdot
\end{eqnarray}

As $\varepsilon_2$ has  a negative norm, so by the  item $(2)$ of Section $3$ of \cite{azizi99unite} we obtain that
$E_{K'}=\langle -1, \varepsilon_2  \rangle$. This in turn implies that $q(K')=Q_{K'}=1$. From which we infer, by Proposition \ref{wada's f.}, that
 $h_2(K')=\frac{1}{2} \cdot 1 \cdot h_2(2) h_2(-p)  h_2(-2p)=\frac{1}{2}     h_2(-p)  h_2(-2p)$. It follows, by the equality \eqref{eq 2}, that

\begin{eqnarray}\label{eq 3}
h_2(L_p)=\frac{1}{4}h_2( L_{p}^+)  h_2(-p)  h_2(-2p)\cdot
\end{eqnarray}

 Note that from  \cite[Theorem 2]{ezrabrown}  and the proof of    \cite[Theorem 1]{ezrabrown}, one deduces easily that    $h_2(-p)=4$ if and only if $\genfrac(){}{0}{2}{p}_4 \not=\genfrac(){}{0}{p}{2}_4$.

 \begin{enumerate}[\rm $\bullet$]
 	\item   Suppose that  $d$ takes the  form  \eqref{ frm i}. So $8$ divides $h_2(-p)$. Note also that,    by  \cite[p. 596]{kaplan1973divisibilitepar8}, $h_2(-2p)$ is divisible by $4$,   and by
 	\cite[Theorem 2]{kuvcera1995parity} $h_2( L_{p}^+)$ is even. It follows  by the equality  \eqref{eq 3}   that  $16$ divides $h_2(L_p)$.
 	So this case is eliminated\\

 	\item  Suppose that  $d$ takes the  form \eqref{ frm ii}. Thus $\genfrac(){}{0}{2}{p}_4=1$ and $\genfrac(){}{0}{p}{2}_4=-1$. So, by  \cite[Theorem 2]{kuvcera1995parity}, $h_2( L_{p}^+)=1$.
 	 Therefore, by the  equality \eqref{eq 3} and  the note below it, we have $h_2(L_p)=h_2(-2p)$.

 	  Keep the notations of \cite[p. 601]{kaplan1973divisibilitepar8} and let $p=a^2+b^2=2e^2-d^2$, with $e\geq 1$.
 	 By using notations in the proof of   \cite[Theorem 1]{ezrabrown} and \cite[p. 601]{kaplan1973divisibilitepar8} we easily deduce that $1=\genfrac(){}{0}{2}{p}_4=(-1)^{b/4}$, so $b\equiv 0\pmod 8$. Therefore, by  \cite[Théroème 3]{kaplan1973divisibilitepar8}, $h_2(-2p)\equiv0\pmod 8$.
 	  It follows, by  \cite[Theorem 2]{leonard1982divisibilityby16}, $h_2(-2p)=8$ if and only if $\genfrac(){}{0}{u}{p}_4=-1$. So  the first item of our theorem.
 	\item  Suppose that	$d$ takes the third form  \eqref{ frm iii}. Without loss of generality we may assume that  $\genfrac(){}{0}{q_1}{q_2}=1$.
  By Proposition \ref{wada's f.} we have
 		\begin{eqnarray*}
 		h(L_{d})&=&\frac{1}{2^5}q(L_{d})h( q_1q_2) h(- q_1q_2)h(2q_1q_2)h(-2q_1q_2)h(2)h(-2)h(-1). 		
 	\end{eqnarray*}
 	By \cite[Corollary 18.4]{connor88} $h_2( q_1q_2)=1$ and by
 		 \cite[pp. 345, 354]{kaplan76} we have $h_2(2q_1q_2)=2$ and  $h_2(-q_1q_2)=4$ respectively.
 		 It is known that $ h(2)=h(-2)=h(-1) =1$. Since by Lemma \ref{lm on unit undex} $q(L_{d})=4$,
 		  then by  passing to the $2$-part in the above equality we get
 		\begin{eqnarray*}
 		h_2(L_{d})=\frac{1}{2^5}.4.1. 4.2.h_2(-2q_1q_2)
 		= h_2(-2q_1q_2).
 	\end{eqnarray*}
 	Therefore, by  \cite[p. 357]{kaplan76} and  Lemma \ref{lm 2-rank = 2}, $\mathrm{Cl}_2(L_{d})=(2, 4)$ if and only if $h_2(-2q_1q_2)=8$ which is  equivalent to  $ \genfrac(){}{0}{-2}{\left|k^2X+lY\right|}=-1$. This achieves the proof.
 	 \end{enumerate}		
 \end{proof}

Let $\ell$ be a positive integer. For a finite abelian group $G$, its $2^\ell$-rank is defined  as $r_{2^\ell}(G)=\mathrm{dim}_{\mathbb{F}_2}(2^{\ell-1}G/2^{\ell}G)$.
Or equivalently looking at the decomposition of the group $G$ into cyclic groups as $G=\prod_{i} C_{n_i}$, the $2^\ell$-rank of $G$ equals the number of $n_i$'s divisible by $2^\ell$.

 \begin{corollary}
 	Let $p$ be a rational prime such that $p\equiv 1\pmod8$ and $\genfrac(){}{0}{2}{p}_4=\genfrac(){}{0}{p}{2}_4=-1$, then the $8$-rank of  $\mathrm{Cl}_2(L_p)$ equals $1$.
 \end{corollary}
 \begin{proof}
 	By \cite[Théorème 10]{taous2008}, we deduce that the $4$-rank of $\mathrm{Cl}_2(L_p)$ equals $1$,  and by the proof of  Theorem \ref{thm (2,4)}, $h_2(L_p)$ is divisible by $16$, so the result.
 \end{proof}
From the proof of Theorem \ref{thm (2,4)} we get
 \begin{corollary}
 	Let $q_1$ and $q_2$ be two rational primes such that $q_1\equiv q_2\equiv 3\pmod 8$, then $h_2(L_{q_1q_2})=h_2(-2q_1q_2)$.
 \end{corollary}

\begin{example} For all the examples below, we used PARI/GP calculator version 2.9.4 (64bit), Dec 20, 2017.
	\begin{enumerate}[\rm 1.]
		\item For $p=89$, $u=17$ and $v=10$, we have $p=u^2-2v^2$,  $\genfrac(){}{0}{2}{p}_4 =-\genfrac(){}{0}{u}{p}_4=1$ and $\mathrm{Cl}_2(L_{89})  \simeq(2, 4)$.
		\item For $q_1=11$, $q_2=19$, $k=1$, $l=3$, $m=-1$, $X=4$ and  $Y=1$, we have  $q_1=l^2-2k^2m$, $2q_2=k^2X^2+2lXY+2mY^2$,
		 $\genfrac(){}{0}{-2}{\left|k^2X+lY\right|} =\genfrac(){}{0}{-2}{7}=-1$ and $\mathrm{Cl}_2(L_{11.19})  \simeq(2, 4)$.
	\end{enumerate}	
\end{example}

 \section{Fields $L_d$ for which $\mathrm{Cl}_2(L_d)$ is of type $(2, 2, 2)$}
 In this section we determine all   fields $L_d$ such that $\mathrm{Cl}_2(L_d)\simeq(2, 2, 2)$. Keep the above notations: $p$, $p_i$, $q$ and $q_i$ are prime   integers satisfying  $p\equiv p_i\equiv1\pmod4$ and $q\equiv q_i\equiv3\pmod4$ with $i\in\mathds{N}^*$.
 From  section 4 of  \cite{ACZ}, it is easy to deduce the following result:
 \begin{lm}\label{lm section 2 lm 1} The rank of the $2$-class group   $\mathrm{Cl}_2(L_d)$ of $L_{d}$ equals $3$  if and only if $d$  takes one of the following forms.
 	\begin{enumerate}[\rm1.]
 		\item $d =p$ with   $ p \equiv1\pmod8$ and   $\genfrac(){}{0}{2}{p}_4= \genfrac(){}{0}{p}{2}_4=1$.
 			\vspace{0.05cm}\item $d=q_1q_2$ with   $q_1\equiv q_2 \equiv7\pmod8$.
 		\item $d=qp$ with   $ q \equiv3\pmod8$  and $p\equiv1\pmod 8$ and $ \genfrac(){}{0}{2}{p}_4=-1$.
 	\vspace{0.07cm}	\item $d=p_1p_2$ with   $ p_1\equiv5\pmod8 $,   $ p_2\equiv1\pmod8$ and $ \genfrac(){}{0}{2}{p_2}_4\neq \genfrac(){}{0}{p_2}{2}_4$.
 		\item $d=q_1q_2p$ with   $ q_1 \equiv q_2\equiv3\pmod8$  and $p\equiv5\pmod{8}.$
 		\item $d=qp_1p_2$ with   $ q \equiv3\pmod8$  and $p_1\equiv p_2\equiv5\pmod{8}.$
 	\end{enumerate}
 \end{lm}

We will need the following lemmas.
 \begin{lm}\label{lm2 on unit undex}
	Let $d=q_1q_2$, with   $q_1\equiv q_2 \equiv7\pmod8$ are two primes such that $\left(\dfrac{	q_1}{	q_2}\right)=1$. Then we have $E_{L_d}=\langle \zeta_{8},\varepsilon_{  2},  \varepsilon_{   q_1q_2} ,\sqrt{\varepsilon_{   2q_1q_2}} \text{ or } \sqrt{\varepsilon_{   q_1q_2}\varepsilon_{   2q_1q_2}}\rangle$, and thus $q(L_d)=4$.
\end{lm}
\begin{proof}
Similar to the proof of Lemma \ref{lm on unit undex}.
\end{proof}

 \begin{lm}\label{lm section 2 lm 2}
 	\begin{enumerate}[\rm 1.]
 		\item Let	$d =p$, for a prime   $ p \equiv1\pmod8$ such that   $\genfrac(){}{0}{2}{p}_4= \genfrac(){}{0}{p}{2}_4=1$. Then, $h_2(L_p)\equiv 0\pmod{16}$, and thus  $\mathrm{Cl}_2(L_p)$ is not elementary.
 		\item 		Let $d=q_1q_2$ with $q_1$ and $q_2$ are two rational primes such that $q_1\equiv q_2\equiv 7\pmod 8$. Then, $h_2(L_{q_1q_2})\equiv 0\pmod{32}$, and thus $\mathrm{Cl}_2(L_{q_1q_2})$ is not elementary.
 		\item Let $d=p_1p_2$ with   $ p_1\equiv5\pmod8 $,    $ p_2\equiv1\pmod8$ are two primes such that  
 		$\genfrac(){}{0}{p_2}{p_1}  = 1$. Then, $h_2(L_{p_1p_2})\equiv 0\pmod{32}$.
 	\end{enumerate}
 \end{lm}
 \begin{proof}
 	\begin{enumerate}[\rm 1.]
 		\item By  equality \eqref{eq 3} in the proof of Theorem \ref{thm (2,4)} we have:
 		\begin{eqnarray*}
 			h_2(L_p)=\frac{1}{4}
 			h_2( L_{p}^+)h_2(-p)h_2(-2p).
 		\end{eqnarray*}
 		Note that by \cite[p. 596]{kaplan1973divisibilitepar8}, $h_2(-2p)$ is divisible by $4$. Since for  a prime $p'\equiv 1\pmod8$,   $h_2(-p')=4$ if and only if  $\genfrac(){}{0}{2}{p'}_4\not= \genfrac(){}{0}{p'}{2}_4$ (see the proof of Theorem \ref{thm (2,4)}), then $h_2(-p)$ is divisible by $8$
 		(in fact $\genfrac(){}{0}{2}{p}_4= \genfrac(){}{0}{p}{2}_4$ and by  \cite[Corollaries 18.4 and 19.6]{connor88} $h_2(-p)$ is divisible by $4$). Thus,
 		$h_2(L_p)$ is divisible by $8\cdot h_2( L_{p}^+)$. By
 		\cite[Theorem 2]{kuvcera1995parity}, $h_2( L_{p}^+)$ is even. From which we infer that
 		 $h_2(L_p)$ is divisible by $16$. So we have the first item.
 		
 		\item  	By Proposition \ref{wada's f.} we have
 		\begin{eqnarray*}
 			h(L_{q_1q_2})=\frac{1}{2^5}q(L_{ q_1q_2})h(q_1q_2) h(-q_1q_2)h(2q_1q_2)h(-2q_1q_2)h(2)h(-2)h(-1)\cdot
 		\end{eqnarray*}
 		On one hand, by \cite[p. 345]{kaplan76}, $h_2(2q_1q_2)$ is divisible by $4$, by   \cite[pp. 354, 356]{kaplan76},  $h_2(-q_1q_2)$ and $h_2(-2q_1q_2)$ are both divisible by    $8$, and, by \cite[Corollary 18.4]{connor88}, $h_2( q_1q_2)=1$. On the other hand,
 		  by Lemma \ref{lm2 on unit undex}, $q(L_{q_1q_2})=4$.  From all these results, it follows by passing to the $2$-part in the above equality that $h_2(L_{d})$ is
 		divisible by $\frac{1}{2^5}\cdot 4\cdot  8\cdot 4\cdot 8 =32$. And the second item follows. 
 		\item We proceed as in  the proof the second item.	
 	\end{enumerate}
 \end{proof}

\begin{example}
Let $d=p$ be as in the first item of the above Lemma.	Using  PARI/GP calculator version 2.9.4 (64bit), Dec 20, 2017, for $3\leq p\leq 10^4$,  we could not find a field $L_d$ such that $h_2(L_{d})=16$. We have the following examples.
	\begin{enumerate}[\rm 1.]
		\item For $p=113$ , we have $\genfrac(){}{0}{2}{p}_4= \genfrac(){}{0}{p}{2}_4=1$ and $h_2(L_{113})=64$.
		\item For $p=337$ , we have  $\genfrac(){}{0}{2}{p}_4= \genfrac(){}{0}{p}{2}_4=1$ and $h_2(L_{337})=32$.
	    \item For $q_1=7$ and $q_2=31$, we have $h_2(L_{7.31})=32$.
		\item For $q_1=7$ and $q_2=23$, we have $h_2(L_{7.23})=64$.
	\end{enumerate}
\end{example}

 \begin{lm}\label{lm section 2 lm 3}
 	\begin{enumerate}[\rm 1.]
 		\item If  $d=q_1q_2p$  with   $ q_1 \equiv q_2\equiv3\pmod8$  and $p\equiv5\pmod{8}$, then $h_2(L_{ q_1q_2p})\equiv 0\pmod{32}$. So $\mathrm{Cl}_2(L_{ q_1q_2p})$ is not elementary.
 		\item 	If   $d=qp_1p_2$ with   $ q \equiv3\pmod8$  and $p_1\equiv p_2\equiv5\pmod{8}$,  then $h_2(L_{  qp_1p_2})\equiv 0\pmod{32}$. So $\mathrm{Cl}_2(L_{qp_1p_2})$ is not elementary.
 	\end{enumerate}
 \end{lm}
 \begin{proof}
 	\begin{enumerate}[\rm 1.]
 		\item 	Consider the following diagram (Figure \ref{ diagram Ld/Qi} below):
 		\begin{center}
 			\vspace*{-0.3 cm}
 			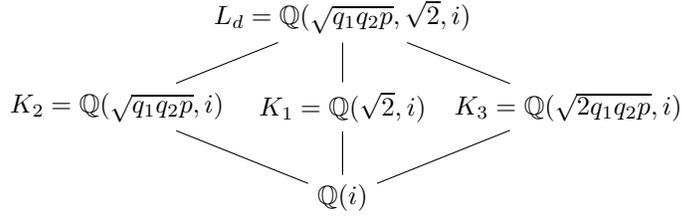
\begin{figure}[H]\label{ diagram Ld/Qi}
 				
 				\begin{minipage}{5cm}
 					{\footnotesize
 						\hspace*{-2.5cm}\begin{tikzpicture} [scale=1.2]
 						\node (Q)  at (0,0) {$\mathbb Q(i)$};
 						\node (d)  at (-2.5,1) {$K_2=\mathbb Q(\sqrt{q_1q_2p}, i)$};
 						\node (-d)  at (2.5,1) {$K_3=\mathbb Q(\sqrt{2 q_1q_2p}, i)$};
 						\node (zeta)  at (0, 1) {$K_1= \mathbb{Q}(\sqrt 2,  i)$};
 						\node (zeta d)  at (0, 2) {$L_d=\mathbb Q(\sqrt{q_1q_2p}, \sqrt 2, i)$};
 						\draw (Q) --(d)  node[scale=0.4, midway, below right]{};
 						\draw (Q) --(-d)  node[scale=0.4, midway, below right]{};
 						\draw (Q) --(zeta)  node[scale=0.4, midway, below right]{};
 						\draw (Q) --(zeta)  node[scale=0.4, midway, below right]{};
 						\draw (zeta) --(zeta d)  node[scale=0.4, midway, below right]{};
 						\draw (d) --(zeta d)  node[scale=0.4, midway, below right]{};
 						\draw (-d) --(zeta d)  node[scale=0.4, midway, below right]{};
 						\end{tikzpicture}}
 				\end{minipage}
 				\caption{$L_{q_1q_2p}/\mathbb Q(i)$.}
 			\end{figure}
 		\end{center}
 	By applying   Proposition \ref{Kurod's f.} to the $V_4$-extension $L_{q_1q_2p}/\mathbb Q(i)$ we get:
 	
 		\begin{eqnarray*}
 		h(L_d)&=&\frac{1}{4}Q(L_d/\mathbb{Q}(i))h(K_1)h(K_2)h(K_3)/h(\mathbb{Q}(i))^2.
 	\end{eqnarray*}
 It is known that $h(K_1)=h(\mathbb{Q}(i))=1$ (in fact,  $K_1=\mathbb{Q}(\zeta_8)$), so		
 		\begin{eqnarray}\label{eq pqq}
 			h(L_d)&=&\frac{1}{4}Q(L_d/\mathbb{Q}(i))h(K_2)h(K_3).
 		\end{eqnarray}
  Note that by \cite[Proposition 2]{mccall1995imaginary},  the ranks of the $2$-class groups of $K_2$ and $K_3$ equal  $2$ and $3$ respectively.
  The author of \cite{azizi99(22} determined all  fields $\mathbb{Q}(\sqrt{d},  i)$ for which the $2$-class group  is of type $(2, 2)$.
 By checking all the results of this last reference,  we deduce that $2$-class group  of $K_2$ is not of type $(2, 2)$. So $8$ divides the class number of $K_2$.

On the other hand,  by   \cite[Théorème 5.3]{azizitaous(24)(222)}   the $2$-class group of $K_3$ is not of type  $(2, 2, 2)$. So the class number of $K_2$ is divisible by $16$.
 Hence   by    equality \eqref{eq pqq} $h(L_d)$ is divisible by $\frac{1}{4}\cdot 8 \cdot 16=32$. So the first item.

 		\item 
 		We similarly prove the second item by using \cite[Théorème 5.3]{azizitaous(24)(222)} and \cite[Proposition 2]{mccall1995imaginary}.
 	\end{enumerate}
 \end{proof}

\begin{example}
Let $d=qp_1p_2$ be as in the second item of the above Lemma.	Using  PARI/GP calculator version 2.9.4 (64bit), Dec 20, 2017, for $3\leq q, p_i\leq 100$,  we could not find a field $L_d$ such that $h_2(L_{d})=32$. We have the following examples.
	\begin{enumerate}[\rm 1.]
		\item For $q_1=3$, $q_2=11$  and $p=5$, we have    $h_2(L_{3.11.5})=32$.
	\item For $q_1=3$, $q_2=11$  and $p=13$, we have    $h_2(L_{3.11.13})=128$.
		\item For  $q=3 $, $p_1=5 $ and  $p_2= 13$, we have    $h_2(L_{3.5.13})=64$.
		\item For  $q=3 $, $p_1=5 $ and  $p_2= 29$, we have    $h_2(L_{3.5.29})=128$.
	\end{enumerate}
\end{example}

 \begin{lm}\label{lm3 on unit undex}
	Let $d=qp$ with   $q \equiv3\pmod8$,  $p\equiv1\pmod 8$ and  $\genfrac(){}{0}{p}{q}=-1$. Then  $E_{L_d}=\langle \zeta_{8}, \varepsilon_{  2},   	\sqrt{ \varepsilon_{ pq}}, \sqrt{ \varepsilon_{   2pq}}\rangle$, and thus $q(L_d)=8$.
\end{lm}
\begin{proof}
	We proceed as in  the proof of Lemma \ref{lm on unit undex}.
\end{proof}

Now we are able to state the main theorem of this subsection.
 \begin{theorem}\label{thm 2 }
 	Let $d$ be a square free integer,  then $\mathrm{Cl}_2(L_d) \simeq(2, 2, 2)$ if and only if $d$ takes one of  the two following forms:
 	\begin{enumerate}[\rm 1.]
 		\item $d=p_1p_2$ with   $ p_1\equiv5\pmod8 $,    $ p_2\equiv1\pmod8$,  $\genfrac(){}{0}{2}{p_2}_4\neq \genfrac(){}{0}{p_2}{2}_4$ and
 		$\genfrac(){}{0}{p_2}{p_1}  =-1$.
 		\item  $d=qp$ with   $ q \equiv3\pmod8$,  $p\equiv1\pmod 8$,  $\genfrac(){}{0}{2}{p}_4 =-1$ and
 		$\genfrac(){}{0}{p}{q} =-1$.
 	\end{enumerate}
 \end{theorem}
 \begin{proof}
 	 	It suffices to determine  which forms of $d$ appearing in Lemma \ref{lm section 2 lm 1},  we have $h_2(L_d)=8$. Let us start, as above,  by eliminating certain inconvenient cases.	
 	\begin{enumerate}[\rm $\bullet$]
 		\item The  forms of $d$ in the first and the second items of Lemma \ref{lm section 2 lm 1} are eliminated by Lemma \ref{lm section 2 lm 2}.
 		\item The   forms of $d$ in the  two last  items of Lemma \ref{lm section 2 lm 1} are eliminated by Lemma \ref{lm section 2 lm 3}.
 	\end{enumerate}
 	 It follows that it suffices  to check the two following cases:
 	\begin{enumerate}[\rm (I)]
 		\item $d=p_1p_2$ with   $ p_1\equiv5\pmod8 $,    $ p_2\equiv1\pmod8$ and $\genfrac(){}{0}{2}{p_2}_4\neq \genfrac(){}{0}{p_2}{2}_4$. \label{frm IS2}
 			\vspace{0.07cm}\item $d=qp$ with   $ q \equiv3\pmod8$  and $p\equiv1\pmod 8$ and $\genfrac(){}{0}{2}{p }_4=-1.$\label{frm IIS2}
 	\end{enumerate}	
 	$\bullet$ Suppose that  $d$ takes the form \eqref{frm IS2},  then  \cite[Theorem 4.2]{azt2018}  and the last assertion of Lemma \ref{lm section 2 lm 2}
 	give the first item.
 	
 	\noindent$\bullet$ Suppose now that  $d$ takes the form  \eqref{frm IIS2}.	Consider the following diagram (Figure \ref{dia2} below):
 	\begin{center}
 		\begin{figure}[H]
 			{\footnotesize
 				\begin{tikzpicture} [scale=1.1]
 				\node (Q)  at (0, 0) {$\mathbb Q(\sqrt{2})$};
 				\node (d)  at (-2.5, 1) {$L_{d}^+=\mathbb Q(\sqrt d, \sqrt 2)$};
 				\node (-d)  at (2.5, 1) {$K_1=\mathbb Q(\sqrt{2}, \sqrt{-d})$};
 				\node (zeta)  at (0, 1) {$K_2= \mathbb{Q}(\sqrt 2,  i)$};
 				\node (zeta d)  at (0, 2) {$L_d=\mathbb Q(\sqrt 2, \sqrt d, i)$};
 				\draw (Q) --(d)  node[scale=0.4, midway, below right]{};
 				\draw (Q) --(-d)  node[scale=0.4, midway, below right]{};
 				\draw (Q) --(zeta)  node[scale=0.4, midway, below right]{};
 				\draw (Q) --(zeta)  node[scale=0.4, midway, below right]{};
 				\draw (zeta) --(zeta d)  node[scale=0.4, midway, below right]{};
 				\draw (d) --(zeta d)  node[scale=0.4, midway, below right]{};
 				\draw (-d) --(zeta d)  node[scale=0.4, midway, below right]{};
 				\end{tikzpicture}}
 			\caption{$L_d/\mathbb Q(\sqrt 2)$.}\label{dia2}
 		\end{figure}
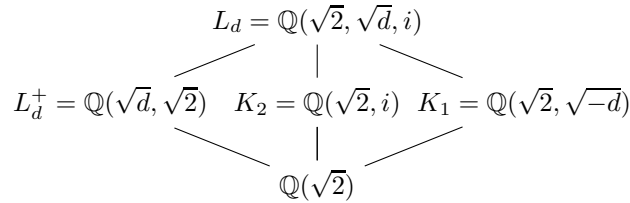
 	\end{center}
 	By    Proposition \ref{ class number f.}  we have:
 	\begin{eqnarray}\label{eq proof 222}
 	 h(L_d)=\frac{Q_{L_d}}{Q_{K_1}Q_{K_2}}\frac{\omega_{L_d}}{\omega_{K_1} \omega_{K_2}}
 	 \frac{h( L_{d}^+)h(K_1)h(K_2)}{h(\mathbb{Q}(\sqrt{2}))^2}\cdot
 	\end{eqnarray}
Note that    $h(K_2)=h(\mathbb{Q}(\sqrt{2}))=1$.
  We have,  $d\varepsilon_2$ is not a square in $\mathbb{Q}(\sqrt{2})$  since otherwise we will get, for some $\alpha$ in $\mathbb{Q}(\sqrt{2})$,  $d\varepsilon_2=\alpha^2$, then
  $N_{\mathbb{Q}(\sqrt{2})/\mathbb{Q}}(d\varepsilon_2)=-d^2=N_{\mathbb{Q}(\sqrt{2})/\mathbb{Q}}(\alpha)^2$. Which is false. It follows  by \cite[Proposition 3]{azizi99unite},  that  $\{\varepsilon_2\}$ is a fundamental system of units of $K_1$,  i.e.,  $ Q_{K_1}=1$.
 	Since  $ Q_{K_2}=1$ (cf. \cite[Lemma 2.5]{ACZ}),  $ \omega_{L_d}= \omega_{K_2}=8$ and $ \omega_{K_1}=2$,   then by passing to the $2$-part in the equality \eqref{eq proof 222},  we get
 	$h_2(L_d)= \frac{Q_{L_d}}{2} h_2( L_{d}^+)h_2(K_1).$ By Proposition \ref{Kurod's f.},  we have $	h_2(K_1)=\frac{1}{2}h_2(-2d)h_2(-d)h_2(2).$ So
 	\begin{eqnarray}\label{eq 6}
 	h_2(L_d)= \frac{Q_{L_d}}{4} h_2( L_{d}^+)h_2(-2d)h_2(-d).
 	\end{eqnarray}

 Note that, by  Lemma \ref{lm3 on unit undex}, if $\genfrac(){}{0}{p}{q } =-1$,  then $Q_{L_d}=1$. Note also that,
 \begin{enumerate}
\item 	from  \cite[Théorème 2]{azmohib} and its proof,  we deduce that $h_2(L_d^+)$ is divisible by $4$ and $h_2(L_d^+)=4$ if and only if $\genfrac(){}{0}{p}{q } =-1$,
\item from  \cite[Corollaries 18.5 and 19.6]{connor88},  we deduce that $h_2(-d)$ is even and $h_2(-d)=2$ if and only if $\genfrac(){}{0}{p}{q } =-1$,
\item from \cite[Corollaries 18.5 and 19.6]{connor88}  and \cite[p. 353]{kaplan76},  we deduce that $h_2(-2d)$ is divisible  by $4$ if and only if $\genfrac(){}{0}{p}{q } =-1$.
\end{enumerate}
Hence plugging all of these results into    equality \eqref{eq 6},  one gets the second item. Which completes the proof.
 \end{proof}

 \begin{example}
 	\begin{enumerate}[\rm 1.]For all the examples below, we used PARI/GP calculator version 2.9.4 (64bit), Dec 20, 2017.
 		\item Let $p_1=29$ and $p_2=17$. We have  $ p_2\equiv1\pmod8$,  $\genfrac(){}{0}{2}{p_2 }_4  \neq \genfrac(){}{0}{p_2}{2 }_4$,
 		$\genfrac(){}{0}{p_2}{p_1 }=-1$ and $\mathrm{Cl}_2(L_{29.17})  \simeq(2, 2, 2)$.
 		
 		\item Let $q=11$ and $p=17$. We have   $\genfrac(){}{0}{2}{p }_4=-1$,
 		$\genfrac(){}{0}{p}{q } =-1$ and $\mathrm{Cl}_2( L_{11.17})  \simeq(2, 2, 2)$.
 	\end{enumerate}
 \end{example}
\section*{Acknowledgment}
We would like to thank the unknown referee  for his/her several helpful suggestions that helped us to improve our paper, and for calling our attention
to the missing details. Many thanks are also due to Professor  Elliot Benjamin for his important remarks on this paper.
 
\end{document}